\documentclass[12pt, twoside]{article}
\usepackage{amsmath,amsthm,amssymb}
\usepackage{times}
\usepackage{enumerate}
\usepackage{mathtools}

\makeatletter
\def\author#1{\gdef\autrun{\def\and{\unskip, }#1}\gdef\@author{#1}}

\makeatother

\def\subjclass#1{{\renewcommand{\thefootnote}{}%
\footnote{\emph{Mathematics Subject Classification (2010):} #1}}}
\def\keywords#1{\par\medskip
\noindent\textbf{Keywords.} #1}

\newtheorem{thm}{Theorem}[section]

\newtheorem*{theorem-non}{Theorem}

\newtheorem{rem}[thm]{Remark}

\newtheorem{lem}[thm]{Lemma}
\newtheorem{pro}[thm]{Proposition}
\newtheorem{cor}[thm]{Corollary}

\newtheorem{nota}[thm]{Notation}

\usepackage{graphicx}
\newcommand{\coverpage}[3]{\thispagestyle{empty}
    \addtocounter{page}{-1}
\null\vspace*{-1cm} \hfill\includegraphics[scale=0.2]{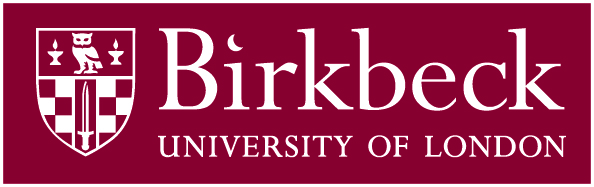} \vskip 2in
\begin{center} \begin{minipage}{0.7\textwidth}\begin{center}\Huge\bf{#1}\end{center} \end{minipage}\end{center}  \vfill
\begin{center} {\large By}\bigskip\\ {\large #2}\\ \end{center} \vfill 
 \framebox{\begin{minipage}{\textwidth}
Birkbeck Pure Mathematics Preprint Series\hfill
Preprint Number #3 \\ \\
\null\hfill www.bbk.ac.uk/ems/research/pure/preprints 
\end{minipage}}
\newpage}

\begin{document}

\coverpage{On a conjecture of Street and Whitehead on locally maximal product-free sets}{Chimere S. Anabanti and Sarah B. Hart}{12}

\title{On a conjecture of Street and Whitehead on locally maximal product-free sets\subjclass{Primary 20D60; Secondary 20P05}}
\author{Chimere S. Anabanti\thanks{The first author is supported by a Birkbeck PhD Scholarship}\\ c.anabanti@mail.bbk.ac.uk \and Sarah B. Hart\\ s.hart@bbk.ac.uk}

\date{May 6, 2015}

\maketitle

\begin{abstract}
\noindent Let $S$ be a non-empty subset of a group $G$. We say $S$ is product-free if $S\cap SS=\varnothing$, and $S$ is locally maximal if whenever $T$ is product-free and $S\subseteq T$, then $S=T$. Finally $S$ fills $G$ if $G^*\subseteq  S \sqcup SS$ (where $G^*$ is the set of all non-identity elements of $G$), and $G$ is a filled group if every locally maximal product-free set in $G$ fills $G$. Street and Whitehead \cite{SW1974} investigated filled groups and gave a classification of filled abelian groups. In this paper, we obtain some results about filled groups in the non-abelian case, including a classification of filled groups of odd order. Street and Whitehead conjectured that the finite dihedral group of order $2n$ is not filled when $n=6k+1$ ($k\geq 1$). We disprove this conjecture on dihedral groups, and in doing so obtain a classification of locally maximal product-free sets of sizes 3 and 4 in dihedral groups, continuing earlier work in \cite{AH2015} and \cite{GH2009}. 

\end{abstract}

\section{Introduction}
A non-empty subset $S$ of a finite group $G$ is called product-free if $xy=z$ does not hold for any $x,y,z \in S$. Equivalently, writing $SS$ for $\{xy: x, y \in S\}$, we have $S \cap SS = \emptyset$. Product-free sets were originally studied in abelian groups, and therefore they are often referred to in the literature as sum-free (or sumfree) sets. If $S$ is product-free in $G$, and not properly contained in any other product-free subset of $G$, then we call $S$ a locally maximal product-free set (see \cite{AH2015}, \cite{GH2009} and \cite{SW1974}). On the other hand, a product-free set $S$ is called maximal if no product-free set in $G$ has size bigger than $|S|$. In the latter direction, see \cite{DY1969}, \cite{GR2005} and \cite{RS1970}. There has been a good deal of work on maximal product-free sets in abelian groups; for example Green and Ruzsa in \cite{GR2005} were able to determine, for any abelian group $G$, the
cardinality of the maximal product-free sets of $G$. Gowers \cite[Theorem 3.3]{gowers}  proved that if the smallest nontrivial
representation of $G$ is of dimension $k$ then $G$ has no product-free sets of size greater than
$k^{−1/3}n$. Much less is known about sizes of locally maximal product-free sets, in particular the minimal size of a locally maximal product-free set. \\

Since every product-free set is contained in a locally maximal product-free set, we can gain information about product-free sets in a group by studying its locally maximal product-free sets. In connection with group Ramsey Theory, Street and Whitehead \cite{SW1974} noted that every partition of a group $G$ (or in fact, of $G^{\ast}$) into product-free sets can be embedded into a covering by locally maximal product-free sets, and hence to find such partitions, it is useful to understand locally maximal product-free sets. They remarked that many examples of these sets have the additional property that $G^{\ast} \subseteq S \cup SS$, and with that in mind gave the following definition. A subset $S$ of a group $G$ is said to {\em fill} $G$ if $G^*\subseteq  S \sqcup SS$. The group $G$ is called a {\em filled group} if every locally maximal product-free set in $G$ fills $G$. Street and Whitehead in \cite{SW1974} and \cite{SW1974A} classified the abelian filled groups and conjectured that the dihedral group of order $2n$ is not filled when $n=6k+1$ for $k\geq 1$. One consequence of our results in this paper is that this conjecture is false.\\

\noindent This paper is aimed at throwing more light on locally maximal product-free sets (LMPFS for short) and filled groups in the non-abelian case. In Section 2 we look at filled groups. We show (Theorem \ref{soluble}) that all non-abelian finite filled groups have even order, and that all finite nilpotent filled groups of even order are 2-groups.  Using GAP \cite{gap} we have seen that for groups of order up to 31 the only examples of non-abelian filled groups are 2-groups or dihedral (see Table~\ref{table1}). Therefore the dihedral case is of interest.  In Section $3$, we study LMPFS in finite dihedral groups and  classify all LMPFS of sizes $3$ and $4$ in dihedral groups. (Groups containing a locally maximal product-free set of size 1 or 2 were classified in \cite{GH2009}.) In Section 4 we look at filled dihedral groups, give a counterexample to the conjecture of Street and Whitehead and obtain some restrictions on the possible orders of filled dihedral groups. \\

In the rest of this section we establish the notation we will need and gather together some useful results. All groups in this paper are finite.
Given a positive integer $n$, we write $C_n=\left\langle x|~x^n=1\right\rangle$ for the cyclic group of order $n$ and $D_{2n}=\langle x,y|~x^{n}=y^2=1, xy=yx^{-1} \rangle$ for the dihedral group of order $2n$ (where $n>1$). In $D_{2n}$ the elements of $\langle x\rangle$ are called rotations and the elements of $\langle x\rangle y$ are called reflections. For any subset $S$ of $D_{2n}$, we write Rot($S$) for $S \cap \langle x \rangle$, the set of rotations of $S$, and Ref($S$) for $S \cap \langle x \rangle y$, the set of reflections of $S$.  Let $S$ and $V$ be subsets of a finite group $G$. We define $SV:=\{sv|~s\in S,v\in V\}$, $S^{-1}:=\{s^{-1}|s\in S\}$, $T(S):=S \cup SS \cup SS^{-1} \cup S^{-1}S$ and $\sqrt{S}:=\{x \in G:x^2 \in S\}$. The following  results will be used repeatedly.

\begin{lem}[Lemma 3.1 of \cite{GH2009}]\label{GH2009L}
Let $S$ be a product-free set in a finite group $G$. Then $S$ is locally maximal if and only if $G=T(S) \cup \sqrt{S}$.
\end{lem}

The following result is well-known but we include a short proof for the reader's convenience. 
\begin{lem} \label{lemma1.2} Let $H$ be a subgroup of a group $G$. Any non-trivial coset of $H$ is product-free in $G$. Further, if $H$ is normal and $Q$ is product-free in $G/H$ then the set $S = \{g \in G: gH \in Q\}$ is product-free in $G$. \end{lem}
\begin{proof}
For the first statement, if for some $h_1, h_2, h_3 \in H$ and $g \in G$ we have $(h_1g)(h_2g) = (h_3g)$, then $g = h_1^{-1}h_3h_2^{-1} \in H$. Therefore if $g \notin H$ we have $(Hg)(Hg) \cap Hg = \emptyset$, so $Hg$ is product-free. Now suppose $H$ is normal with $Q$ and $S$ as defined in the statement of the lemma. Then $SS = \{a \in G: aH \in QQ\}$. The fact that $S$ is product-free now follows immediately from the fact that $Q$ is product-free. 
\end{proof}

The following is a straightforward consequence of the definitions. 

\begin{pro}\label{AH2015A_Pro}
Each product-free set of size $\frac{|G|}{2}$ in a finite group $G$ is the non-trivial coset of a subgroup of index 2. Furthermore such sets are locally maximal and fill $G$. 
\end{pro}


\section{Filled groups}\label{sec3}
Street and Whitehead in \cite{SW1974} and \cite{SW1974A} investigated locally maximal product-free sets properties in some groups. They proved the following results.

\begin{lem}\cite[Lemma 1]{SW1974} \label{swl1} Let $G$ be a finite group and $N$ a normal subgroup of $G$. If $Q$ is a locally maximal product-free set in $G/N$ that does not fill $G/N$, then the set $S$ given by $S = \{g: gN \in Q\}$ is a locally maximal product-free in $G$ that does not fill $G$. That is, if $G$ is filled then $G/N$ is filled. \end{lem}

\begin{thm}\cite[Theorem 2]{SW1974} A finite abelian group is filled if and only if it is $C_3, C_5$ or an elementary abelian 2-group.
\end{thm}

They also observed the following. 

\begin{lem}\label{co3}
If $G$ is a filled group with a normal subgroup of index 3, then $G \cong C_3$.
\end{lem}

\begin{proof} Let $N$ be a normal subgroup of index 3 and $S$ be a nontrivial coset. Then $S$ is product-free by Lemma \ref{lemma1.2}. Moreover $S \cup SS \cup SS^{-1} = G$. Therefore $S$ is a locally maximal product-free set by Lemma~\ref{GH2009L} and so $S$ must fill $G$, which implies that $G^{\ast} \subseteq S \cup SS$. But $S \cup SS = G  -  N$. Therefore $N = \{1\}$ and $G$ is cyclic of order 3. \end{proof}

Street and Whitehead also observed that Lemma \ref{swl1} implies that the quotients of any filled non-abelian group $G$ must themselves be filled. In particular, the quotient of $G$ by its derived group $G'$ must be either an elementary abelian 2-group or cyclic of order 5 (it cannot be cyclic of order 3 by Lemma \ref{co3}). These conditions are not sufficient. The counterexamples given in \cite{SW1974} are $D_{14}$ (which in fact {\em is} a filled group, as we shall show), the quaternion group of order 8 and the alternating group of degree 5. \\

Our main aim in this section is to classify filled groups of odd order. We begin with $p$-groups of odd order.

\begin{pro}\label{4.12}
Suppose $G$ is a finite $p$-group, where $p$ is an odd prime. Then $G$ is filled if and only if $G$ is either $C_3$ or $C_5$.
\end{pro}

\begin{proof}
Certainly $C_3$ and $C_5$ are filled. For the reverse implication, let $G$ be a finite $p$-group of order $p^n$. We proceed by induction on $n$. If $G$ is  non-abelian, then the quotient of $G$ by its centre $Z(G)$ is a strictly smaller $p$-group so, inductively, is either $C_3$ or $C_5$ (since $p$ is odd). But it is a basic result that if $G/Z(G)$ is cyclic, then $G$ is abelian, giving a contradiction. Therefore $G$ is abelian, and now the result follows immediately from the classification of filled abelian groups.
\end{proof}

The next theorem, Theorem \ref{soluble}, makes use of an observation in \cite{SW1974}.  Theorem 3 of that paper asserts that if $G$ is a finite nonabelian filled group, then either $G = G'$ or $G/G'$ is an elementary abelian 2-group, or $G/G' \cong C_5$ and $|G|$ is even. The proof given is that since $G/G'$ must be a filled abelian group, it is either trivial, or elementary abelian 2-group, or $C_3$ or $C_5$. Now $C_3$ is impossible by Lemma \ref{co3}. So if $G$ has odd order, we must have that $G/G'$ is cyclic of order 5. A set is then described, based on an element $a$ of $G - G'$, which the authors claim is locally maximal product free but does not fill $G$. But in fact the given set is only locally maximal if $a$ has order 5. The existence of such an element is not guaranteed when $G/G'$ is cyclic of order 5, even if $G$ has odd order. We are grateful to Robert Guralnick for providing us with an example of a group without such an element --- the group is an extension of an extraspecial group of order $5^{11}$ by the Frobenius group of order $55$, such that the fifth power of each element of order 5 in the Frobenius group is a central element of order 5 in the extraspecial group. In this case the derived group has index 5 and contains all elements of order 5. We resolve that issue in the following lemma and theorem by reducing to a situation where we can be certain of the existence of the required element. 

\begin{lem}\label{order5}
Suppose $G$ is a finite group with a normal subgroup $N$ of index five, such that not every element of order five in $G$ is contained in $N$. If $G$ is filled, then $G$ is cyclic of order 5.
\end{lem}

\begin{proof} Our argument is based on the construction given in \cite{SW1974}. If $N$ is trivial then $G \cong C_5$ and $G$ is filled. If $|N| = 2$, then $G \cong C_{10}$ and $G$ is not filled. So we may assume $N$ has order at least three.
 Let $h$ be an element of order 5 in $G$ with $h \notin N$. Then let $S = \{h\} \cup h^2N^{\ast}$ (where $N^{\ast}$ is the set of nonidentity elements of $N$). Then $SS = \{h^2\} \cup h^3N^{\ast} \cup h^4N$. (The fact that $(h^2N^{\ast})^2 = h^4N$ follows because $|N| > 2$.) So $S$ is product-free, but does not fill $G$. Now $SS^{-1}  = h^4N^{\ast} \cup hN^{\ast} \cup N$. Thus $T(S) = G  -  \{h^3\}$. Since $h^3 \in \sqrt S$, we can now conclude that $G = T(S) \cup \sqrt S$, which means $S$ is a locally maximal product-free set that does not fill $G$. So $G$ is not filled.
\end{proof}

\begin{thm}\label{soluble}
The only filled groups of odd order are $C_3$ and $C_5$.  
\end{thm}

\begin{proof}
Let $G$ be a nontrivial group of odd order. We proceed by induction on the order of $G$. Groups of order 3 and 5 are filled, so assume $|G| > 5$, and, inductively, that if $H$ is a filled group of odd order with $|H| < |G|$, then $H$ is isomorphic to either $C_3$ or $C_5$. \\

If $G$ is abelian, then $G$ is not filled. So we may assume that $G$ is nonabelian. Then, because $G$ is soluble, the derived group $G'$ is a proper nontrivial normal subgroup of $G$. Therefore $G/G'$ is a filled group of order less than $|G|$, and   hence is isomorphic to either $C_3$ or $C_5$. However, as we have noted, if $S$ is any non-trivial coset of a normal subgroup of index 3, then $S$ is a locally maximal product-free set that does not fill $G$. Therefore $G/G'$ is cyclic of order 5. If $G''$ is nontrivial, then we can apply the same argument to $G/G''$, which would imply that $G/G''$ is also cyclic of order 5, and thus that $G'' = G'$, contradicting the solubility of $G$. Therefore $G'' = \{1\}$ and $G'$ is abelian.\\

Since $G$ has order greater than 5, Proposition \ref{4.12} implies that $G$ is not a $p$-group. Thus there is at least one prime $p$, with $p \neq 5$, dividing the order of $G$. Any Sylow $p$-subgroup $K$ of $G'$ is also a Sylow $p$-subgroup of $G$. But $G'$ is normal in $G$, and abelian, whence $K$ is normal in $G$. Now $G/K$ is filled, meaning that $G/K$ has order 5, which implies $K = G'$.  Therefore $5$ does not divide the order of $|K|$, which means there are elements of order 5 in $G$ that do not lie in $G'$ (in fact of course all elements of order 5 lie outside of $G'$). Therefore, by Lemma \ref{order5}, $G$ is not filled. The result now follows by induction.
%
%
%
\end{proof}

Groups of even order are of course less amenable to analysis. We have the following step in this direction.

\begin{lem}\label{interesting}
If $G$ is a filled nilpotent group, then $G$ is either a 2-group or isomorphic to $C_3$ or $C_5$.
\end{lem}

\begin{proof}
Let $G$ be a filled nilpotent group. If $G$ has odd order then $G$ is either $C_3$ or $C_5$ by Theorem~\ref{soluble}. So assume $G$ has even order. Then $G$ is a direct product of $p$-groups (its Sylow subgroups), and its Sylow 2-subgroup $N$ is nontrivial. The quotient $G/N$ (which must be filled) is isomorphic to the direct product of the remaining Sylow subgroups, which is a group of odd order. If $N \neq G$ then $G/N$ is either $C_3$ or $C_5$. We know no filled group can have a normal subgroup of index 3, so $G/N$ must be cyclic of order 5, and clearly all elements of order 5 in $G$ lie outside $N$. Thus, by Lemma \ref{order5}, $G$ is not filled.  Therefore $N = G$. That is, if $G$ is a filled nilpotent group then $G$ is either a 2-group or isomorphic to $C_3$ or $C_5$. 
\end{proof}

In the light of Lemma \ref{interesting} it would be interesting to have a classification of filled 2-groups, as this would enable a full classification of filled nilpotent groups. We will show in Section 4 that  $D_8$ is the only filled nonabelian dihedral 2-group. We can eliminate generalised  quaternion groups from our enquiries now. For a positive integer $n$, with $n>1$, the generalised  quaternion group of order $4n$ is the group $Q_{4n} = \langle a, b: a^{2n} = 1, b^2 = a^n, ba = a^{-1}b\rangle$. We have the following.

\begin{pro}
No generalised  quaternion group is filled.
\end{pro}

\begin{proof}
Let $G$ be generalised  quaternion. Then $G$ has a cyclic subgroup $N$ of index 2, and $G$ contains a unique involution $z$. Let $S$ be any locally maximal product-free set of $N$  containing $z$. Then because every element of $G - N$ is a square root of $z$, we have that $S$ is locally maximal product-free in $G$. But $S$ clearly does not fill $G$.
\end{proof}

The only nonabelian filled groups we know of that are not 2-groups are dihedral (see Table \ref{table1} for a complete list of the filled groups of order less than 32).  Therefore it makes sense to study dihedral groups a little more carefully. This is the object of the next section. 

\section{Locally maximal product-free sets in dihedral groups}
\begin{thm}\label{CC1}\label{AH2015A_T}
Let $S$ be a locally maximal product-free set in a finite dihedral group $G$ of order $2n$. Then $|G|\leq |S|^2+|S|$.
\end{thm}
\begin{proof}
Suppose $S$ is a locally maximal product-free set of size $m\geq 1$ in a finite dihedral group $G$. Let $A = $Rot($S$) and $B = $Ref($S$). By the relations in the dihedral group, $BA = A^{-1}B$. We also have that $AA^{-1} = A^{-1}A$ and $B^{-1} = B$. Therefore \begin{align}
T(S) &= S \cup SS \cup SS^{-1} \cup S^{-1}S \nonumber\\
&= A \cup B \cup AA \cup AB \cup A^{-1}B \cup BB \cup AA^{-1}. \label{eq1}
\end{align}
Now, since $G = T(S) \cup \sqrt S$, and $\sqrt S$ cannot contain involutions, it must be the case that Ref($G$) is contained in $T(S)$. That is, we must have \begin{equation}
\mathrm{Ref}(G)  = B \cup AB \cup A^{-1}B = B \cup (A \cup A^{-1})B.
\label{eq2}
\end{equation}
If $|A| = k$ and $|S| = m$ we see that $|G| \leq 2(2k+1)(m-k) = 2m + 2(2m-1)k - 4k^2$. This is a quadratic expression in $k$ which attains its maximum value over all $k$ when $k = \frac{2m-1}{4}$, so attains its maximum value over integers $k$ at either $k = \frac{m-1}{2}$ or $k = \frac{m}{2}$. Substituting either value for $k$ into $2(2k+1)(m-k)$ gives $m(m+1)$. We conclude that $|G| \leq  |S|^2 + |S|$.
\end{proof}

\begin{rem}\label{RSS}
It follows from the proof of Theorem \ref{CC1} that if $S$ is a locally maximal product-free set in a finite dihedral group $G$, then Ref($G$)$=$Ref($T(S)\cup \sqrt{S}$)$=$Ref($S\sqcup SS$). In particular, $S$ must contain at least one reflection.
\end{rem}

We also remark that nearly all other known upper bounds for the order of a finite group $G$ containing a locally maximal product-free set $S$ are in terms of $|\langle S\rangle|$ rather than $|S|$. See \cite{GH2009}, for example. Theorem \ref{CC1} is thus a useful concrete bound for dihedral groups. The only other known upper bounds for the order of a finite group $G$, in terms of the size of a locally maximal product-free set $S$ in $G$, are for the case where $S\cap S^{-1} = \emptyset$, in which case $|G| \leq 4|S|^2 + 1$ \cite[Corollary 3.10]{GH2009}, and the case where $G$ is cyclic, in which case it is easy to see from Lemma \ref{GH2009L} that $|G| \leq \frac{1}{2}\langle 3|S|^2 + 5|S| + 2\rangle$.\\ 

From Theorem \ref{AH2015A_T}, it is clear that if $n > 1$, there is no locally maximal product-free set of size $1$ in $D_{2n}$, and that any locally maximal product-free set of size $2$ must appear in $D_4$ or $D_6$. A simple check shows that any LMPFS of size $2$ must be automorphic to $\{x,y\}$ in $D_4$ and $D_6$, where $x$ is an element of order $n$, and $y$ is a reflection. We next look at sets of size 3 and 4. Locally maximal product-free sets of size $3$ have been classified in \cite{AH2015}, building on work in \cite{GH2009}, but for completeness we include the result for dihedral groups here with a brief proof. As no full classification has been given for size $4$, the one here is a  step in that direction. For the rest of this section, $G$ will be a finite dihedral group of order $2n$ for $n\geq 3$.
\begin{nota}
Where $n$ is even, we denote the non-identity cosets of the maximal subgroups of $G$ by $M_1$, $M_2$ and $M_3$, where $M_1 =$ Ref($G$), $M_2= \{x,x^3,\cdots,x^{n-1},y,x^2y,\cdots,x^{n-2}y\}$ and \linebreak $M_3 = \{x,x^3,\cdots,x^{n-1},xy,x^3y,\cdots,x^{n-1}y\}$ 
respectively. If $n$ is odd, then $M_1$ is the only such coset.
\end{nota}


\begin{thm}\label{R1}
If $S$ is a LMPFS of size $3$ in a finite dihedral group $G$, then $G=D_6$ or $D_8$. Furthermore, up to automorphisms of $G$, there is only one such set; viz. $\{y,xy,x^2y\}$ or $\{x^2,y,xy\}$ according as $G=D_6$ or $D_8$.
\end{thm}

\begin{proof}
By Theorem \ref{AH2015A_T} and the fact that a LMPFS of size $3$ cannot be contained in a group of order less than $6$, we have that $6 \leq |G| \leq 12$. By Proposition \ref{AH2015A_Pro}, the only LMPFS of size 3 in $D_6$ is $M_1$. So assume $|G| \geq 8$. Now $S$ contains at least one reflection by Remark \ref{RSS}, and $S$ contains at least one rotation, because otherwise it would be properly contained in the product-free set $M_1$. Suppose $S$ consists of a rotation $a$, and reflections $b_1$ and $b_2$. If $G=D_8$ and $a$ has order $4$ in $G$, then no such $S$ exist (since any such $S$ is either contained in $M_2$ or $M_3$, or not product-free); if $a$ is the unique involution, then $S$ must be mapped by an automorphism of $G$ into $\{x^2,y,xy\}$ since $b_1$ and $b_2$ must be in distinct conjugacy classes of $G$ for such $S$ to exist. If $G=D_{10}$, then by adjoining $a^{-1}$ to $S$, we get a bigger product-free set that contains $S$; thus no such $S$ exist.
Finally, suppose $G=D_{12}$.
If $a$ is the unique involution, then by Equation \eqref{eq2} $|\mathrm{Ref}(G)| = 4$, a contradiction. Suppose $\circ(a)=3$. Observe that Rot($T(S)\cup\sqrt{S}$)$\subseteq(\{1,a,a^{-1},\sqrt{a}\} \cup \{b_1b_2,b_2b_1\})$. Since $\sqrt{a}$ consists of $a^{-1}$ and an element of order $6$, and $\circ(b_1b_2)=\circ(b_2b_1)$, such $S$ cannot exist. Now, suppose $\circ(a)=6$. As Rot($T(S)\cup\sqrt{S}$)$\subseteq(\{1,a,a^2\} \cup \{b_1b_2,b_2b_1\})$, we have that $|$Rot$(T(S)\cup\sqrt{S})|\leq 5<6=|$Rot$(D_{12})|$; thus no such $S$ exist. Similar arguments show that there is no locally maximal product-free set made up of exactly two rotations and one reflection in $D_8$, $D_{10}$ and $D_{12}$.
\end{proof}

\begin{pro}\label{MR3}\label{AH2015A_P}
Let $G$ be a dihedral group of order $2p$, $p>3$ and prime. If $S$ is a locally maximal product-free set of size $4$ in $G$, then $p$ is either 5 or 7 and $S$ contains exactly two non-identity rotations and two reflections.
\end{pro}
\begin{proof}
By Theorem \ref{CC1}, either $G$ is at least one of $D_{10}$ or $D_{14}$, or no such $G$ exists. Let $G$ be either $D_{10}$ or $D_{14}$. Suppose for a contradiction that $S$ does not contain exactly two non-identity rotations and two reflections. By Remark \ref{RSS}, $S$ contains at least one reflection. At the other extreme, if every element of $S$ is a reflection then $S$ is properly contained in $M_1$, so $S$ is not locally maximal. Therefore $S$ contains at least one rotation. Suppose $S$ contains $1$ non-identity rotation (say $x^i$) and three reflections. Then a quick check shows that $S \cup \{x^{-i}\}$ is also product-free, contradicting the maximality of $S$. Similarly, if $S$ contains three non-identity rotations and one reflection, then there exists a rotation in $S$ whose inverse is not in $S$, and by adjoining this inverse to $S$, we again obtain a contradiction to the locally maximal condition on $S$. Therefore $S$ contains exactly two reflections and two rotations. 
\end{proof}
\begin{cor}\label{MR5}\label{AH2015A_C}
Suppose $G$ is a dihedral group of order $2p$ ($p>3$ and prime). If $S$ is a LMPFS of size $4$ in $G$, then $p$ is either 5 or 7, and $S$ is mapped by an automorphism of $G$ into $\{x^2,x^3,y,x^{-1}y\}$.
\end{cor}

\begin{pro}\label{P1}
If $S$ is a locally maximal product-free set of size $4$ containing exactly four involutions in a dihedral group $G$ such that $10\leq |G|\leq 20$, then $G$ can only be $D_{12}$. Furthermore, $S$ is mapped by an automorphism of $D_{12}$ into $\{x^3,y,xy,x^2y\}$.
\end{pro}
\begin{proof}
 As $M_1$ contains all reflections in $G$, in order for $S$ not to be properly contained in $M_1$, we must have that $\frac{|G|}{2}$ is even and that $S$ contains the unique involution $z$ in the rotation subgroup. But now by Equation \eqref{eq2} we obtain $|\mathrm{Ref}(G)| \leq 6$. Therefore the only possibility is $D_{12}$. So assume $G = D_{12}$ and let $S=\{z,b_1,b_2,b_3\}$, where $b_1,b_2$ and $b_3$ are reflections and $z = x^3$.  If $S$ is locally maximal product-free in $D_{12}$, then by Remark \ref{RSS}, Ref($D_{12}$)$=$Ref($S\sqcup SS$)$\subseteq$ $\{b_1,b_2,b_3,x^3b_1,x^3b_2, x^3b_3\}$. If any two elements of Ref($S\sqcup SS$) are equal, then $S$ is not locally maximal in $D_{12}$.
For no two elements of Ref($S\sqcup SS$) to be equal, we must have that $S$ is of the form $\{x^3,x^iy,x^{i+1}y,x^{i+2}y\}$ for $i=0,1,2,3,4$ or $5$. Thus, the only possible choices are $S:=\{x^3,y,xy,x^2y\}$, $S_1:=\{x^3,xy,x^2y,x^3y\}$, $S_2:=\{x^3,x^2y,x^3y,x^4y\}$, $S_3:=\{x^3,x^3y,x^4y,x^5y\}$, $S_4:=\{x^3,y,x^4y,x^5y\}$ and $S_5:=\{x^3,y,xy,x^5y\}$. By Lemma~\ref{GH2009L}, $S$ is locally maximal product-free in $D_{12}$. As the automorphism $\phi_i: x \mapsto x, y\mapsto x^iy$ maps $S$ into $S_i$ for each $i\in [0,5]$, we are done.
\end{proof}

\begin{lem}\label{L2}
Suppose $S$ is a locally maximal product-free set of size $4$ consisting of three involutions and one non-involution in a dihedral group $G$ such that $10\leq |G|\leq 20$. If $G=D_{12}$, then $S$ is automorphic to $\{x^2,x^3,y,x^5y\}$. Moreover, no such $S$ exist if $|G|\neq 12$. 
\end{lem}
\begin{proof} The argument splits into two cases: Case I where the involutions are all reflections, and Case~II where $\frac{|G|}{2}$ is even and one of the involutions is the central rotation $x^{|G|/4}$.\\

Case I: Let $S=\{a,b_1,b_2,b_3\}$, where $a$ is the non-involution and $b_1,b_2$ and $b_3$ are reflections.
If $a$ has order at least $4$, then $S\cup\{a^{-1}\}$ is product-free, contradicting the local maximality of $S$. Thus $a$ has order $3$ and $G$ is either $D_{12}$ or $D_{18}$. However, no such combination (with any three reflections) gives a locally maximal product-free set of size $4$. For example in $D_{12}$, if $b_1 = x^iy$, $b_2 = x^jy$ and $b_3 = x^ky$ then at least two of $i, j$ and $k$  must have the same parity. But that implies $a \in SS$, a contradiction.  
\\

Case II: Let $S=\{a,b_1,b_2,z\}$, where $a$, $b_1$ and $b_2$ are as in Case I, and $z$ is the unique involution in the rotation subgroup. Here $G$ is $D_{12}$, $D_{16}$ or $D_{20}$. A quick calculation using the fact that \linebreak $G = T(S) \cup \sqrt S$ shows that $$\mathrm{Rot}(G) = \sqrt S \cup\{1, a, a^2, z, az, a^{-1}z, b_1b_2, b_2b_1\}.$$
Since $b_1b_2 = (b_2b_1)^{-1}$, and $S$ is product-free, we have that $a^{-1} \in \sqrt S \cup \{a^2, az\}$. If $a^{-1} \in \sqrt z$ or $a^{-1} = az$ then $a^2 = z$, contradicting the fact that $S$ is product-free. Thus $a^{-1} \in \sqrt a \cup a^2$, which implies $\circ(a) = 3$, and so $G = D_{12}$. By Lemma \ref{GH2009L}, the product-free set $\{x^2,x^3,y,x^5y\}$ is locally maximal. A careful check shows that any other arising LMPFS must be mapped by an automorphism of $D_{12}$ into $\{x^2,x^3,y,x^5y\}$.
\end{proof}

\begin{pro}\label{P3}
Let $S$ be a locally maximal product-free set of size $4$ consisting of two involutions and two non-involutions in a dihedral group $G$ such that $10\leq |G|\leq20$. Then, up to automorphism,  $S$ and $G$ are given in the table below.
\begin{center}
\begin{tabular}{|p{0.5cm}|p{5.2cm}|}\hline
$G$ &  $\qquad \qquad S$\\ \hline
$D_{10}$ & $\{x^2,x^3,y,x^4y\}$\\ \hline
$D_{12}$ & $\{x,x^5,y,x^3y\}, \{x,x^4,y,x^3y\}$ \\ \hline
$D_{14}$ & $\{x^2,x^3,y,x^6y\}$\\ \hline
$D_{16}$ & $\{x^2,x^3,y,x^7y\}, \{x,x^6,y,x^4y\}$\\ \hline
$D_{18}$ & $\{x^2,x^5,y,x^8y\}$\\ \hline
$D_{20}$ & $\{x,x^8,y,x^5y\}$\\ \hline
\end{tabular}
\end{center}
\end{pro}
\begin{proof}
Case I: Let $S=\{a_1,a_2,b_1,b_2\}$, where $a_1,a_2$ are non-involutions, and $b_1,b_2$ are reflections. Assume for the moment that $|G| > 12$. If $a_2 = a_1^{-1}$, then by Equation \eqref{eq2} $|\mathrm{Ref}(G)| \leq 6$, which contradicts our assumption. Thus $a_2 \neq a_1^{-1}$.  In the case of $D_{18}$, this means that at least one of $a_1$ and $a_2$ has order 9. If $\circ(a_1)=9=\circ(a_2)$, with $a_2 \neq a_1^{-1}$, then $S$ is automorphic to $\{x^2,x^5,y,x^8y\}$, which is locally maximal product-free.  Next suppose $G$ is $D_{16}$ or $D_{20}$. If $\circ(a_1)=3$ and $\circ(a_2) = 9$ then a quick check shows that $S$ is not locally maximal product-free. Suppose $\circ(a_1)=\frac{|G|}{4}=\circ(a_2)$. $S$ is not feasible in $D_{16}$ since then $a_2 = a_1^{-1}$. On the other hand, such $S$ is not also possible in $D_{20}$ as Ref($T(S)\cup \sqrt{S}$)$\subseteq \{b_1,b_2,a_1b_1,a_1b_2,b_1a_1,b_2a_1,a_2b_1,a_2b_2,b_1a_2,b_2a_2\}$, and either $a_2=a_1^{-1}$ or $\{a_1,a_2\}$ is not product-free. 
Suppose $\circ(a_1)=\frac{|G|}{2}=\circ(a_2)$. Then Rot($T(S)\cup \sqrt{S})\subseteq $Rot($T(S))\subseteq \{1,a_1,a_2,a_1a_2,a_1^2,a_2^2,b_1b_2,b_2b_1,a_1^{-1}a_2,a_2^{-1}a_1\}$. So the only possible odd powers of a generator of $C_{\frac{|G|}{2}}$ are $a_1,a_2,b_1b_2$ and $b_2b_1$. If $G=D_{20}$, then no such $S$ exists. On the other hand, if $G=D_{16}$, then as $b_2b_1=(b_1b_2)^{-1}$, we must have that $a_2={a_1}^{-1}$ which leads to the conclusion that Ref($T(S)\cup \sqrt{S}$)$\subseteq \{b_1,b_2,a_1b_1,a_1b_2,a_2b_1,a_2b_2\}$, a contradiction. Finally, suppose $\circ(a_1)=\frac{|G|}{4}<\frac{|G|}{2}=\circ(a_2)$. The set $S$ is locally maximal product-free by Lemma \ref{GH2009L}. A careful check shows that any such set must be mapped by an automorphism of the group into $\{x^2,x^3,y,x^7y\}$ or $\{x,x^6,y,x^4y\}$ if $G=D_{16}$, and $\{x,x^8,y,x^5y\}$ if $G=D_{20}$. Now suppose $G=D_{12}$. If $\circ(a_1)=3=\circ(a_2)$, then no such $S$ exist since $a_2=a_1^{2}$. If $\circ(a_1)=6=\circ(a_2)$, then $a_2=a_1^{-1}$. Such $S$ exists, and must be mapped by an automorphism of $D_{12}$ into $\{x,x^5,y,x^3y\}$. If $\circ(a_1) \neq \circ(a_2)$. Any resulting product-free set is locally maximal, and must be mapped by an automorphism of $G$ into $\{x,x^4,y,x^3y\}$. Finally, in the case $G=D_{10}$ or $D_{14}$, the result follows from Corollary \ref{AH2015A_C}.
\\

Case II: Let $S=\{a_1,a_2,b,z\}$, where $a_1,a_2$ are non-involutions, $b$ is a reflection and $z$ is the unique involution in Rot($G$). As $|$Ref($T(S)\cup \sqrt{S}$)$|\leq |\{zb,a_1b,a_2b,a_1^{-1}b,a_2^{-1}b\}|$, $|G|\leq 10$. So $S$ can only exist in $D_{10}$. By Proposition \ref{AH2015A_P}, no such $S$ exist in $D_{10}$.
\end{proof}

\begin{lem}\label{L4}
There is no locally maximal product-free set of size $4$ consisting of at most one involution in a finite dihedral group.
\end{lem}
\begin{proof}
Any locally maximal product-free set $S$ of size $4$ must contain at least one reflection; otherwise $T(S) \cup \sqrt{S} \subseteq \langle S \rangle$ which is cyclic. Now, suppose $S=\{a_1,a_2,a_3,b\}$, where $a_1,a_2$ and $a_3$ are non-involutions, and $b$ is a reflection. As $|$Ref($T(S)\cup \sqrt{S}$)$|\leq |\{a_1b,a_2b,a_3b,$ $a_1^{-1}b,a_2^{-1}b,a_3^{-1}b\}|=6$, and a LMPFS of size $4$ cannot be contained in a group of order less than $8$, we must have that $8\leq |G|\leq 12$. By Propositions \ref{AH2015A_Pro} and \ref{AH2015A_P}, no such $S$ exists in $D_{8}$ and $D_{10}$ respectively. As no three non-involutions can form a locally maximal product-free set in $C_6$, no such $S$ exist in $D_{12}$.
\end{proof}

\noindent We are now in a position to classify all locally maximal product-free sets of size 4 in dihedral groups. 
\begin{thm}\label{R2}
Suppose $S$ is a LMPFS of size $4$ in a dihedral group $G$. Then up to automorphisms of $G$, the possible choices are given as follows:
\begin{center}
\begin{tabular}{|p{0.5cm}|p{10.5cm}|}\hline
$|G|$ &  $\qquad \qquad \qquad \quad S$\\ \hline
$8$ & $\{y,xy,x^2y,x^3y\}, \{x,x^3,y,x^2y\}$\\ \hline
$10$ & $\{x^2,x^3,y,x^4y\}$\\ \hline
$12$ & $\{x^3,y,xy,x^2y\}, \{x^2,x^3,y,x^5y\}, \{x,x^5,y,x^3y\}, \{x,x^4,y,x^3y\}$ \\ \hline
$14$ & $\{x^2,x^3,y,x^6y\}$\\ \hline
$16$ & $\{x^2,x^3,y,x^7y\}, \{x,x^6,y,x^4y\}$\\ \hline
$18$ & $\{x^2,x^5,y,x^8y\}$\\ \hline
$20$ & $\{x,x^8,y,x^5y\}$\\ \hline
\end{tabular}
\end{center}
\end{thm}

\begin{proof}
By Theorem \ref{AH2015A_T} and the fact that a locally maximal product-free set of size $4$ cannot be contained in a group of order less than $8$, we must have that $8\leq|G|\leq 20$. If $G=D_8$, then by Proposition \ref{AH2015A_Pro}, either $S=\{y,xy,x^2y,x^3y\}$ or it is mapped by an automorphism of $G$ into $\{x,x^3,y,x^2y\}$. Now, suppose $10\leq |G|\leq 20$. The result follows from Proposition \ref{P1}, Lemma \ref{L2}, Proposition \ref{P3} and Lemma \ref{L4}.
\end{proof}

\section{Filled dihedral groups}

In this section we obtain some facts about filled dihedral groups. 
In \cite{SW1974} the authors asserted that the dihedral group  of order $2n$ is not a filled group for $n=6k+1$. They went further to produce a locally maximal product-free set ($S:=\{x^{2k+1},\dots,x^{4k},x^{2k+1}y,\dots,x^{4k}y\}$) which they claim does not fill $D_{2n}$. However we have the following.

\begin{pro}\label{P5}
Let $G$ be a dihedral group of order $2n$ for $n=6k+1$ and $k\geq 1$. Then the set $S:=\{x^{2k+1},\dots,x^{4k},x^{2k+1}y,\dots,x^{4k}y\}$ is product-free but not locally maximal in $G$.
\end{pro}

\begin{proof}
The fact that $S$ is product-free follows from our proof since every subset of a product-free set is product-free. So, we only show that $S$ is not locally maximal. To do this, we show that the set $V:=\{x^{2k+1},\dots,x^{4k},x^{2k}y,x^{2k+1}y,\dots,x^{4k}y\}$, which properly contains $S$, is product-free. (One may also do same using $U:=\{x^{2k+1},\dots,x^{4k},x^{2k+1}y,$ $\dots,x^{4k}y,x^{4k+1}y\}$.) Let $A = $ Rot($V$) and $B = $ Ref($V$). We note that $V = V^{-1}$ and so $BA = A^{-1}B = AB$. Therefore $VV = AA \cup AB \cup BA \cup BB = AA \cup BB \cup AB$. Thus $$VV = \{1,x,\cdots,x^{2k}\}\sqcup \{x^{4k+1},x^{4k+2},\cdots,x^{6k}\} \sqcup \{y,xy,\cdots,x^{2k-1}y\}\sqcup \{x^{4k+1}y,x^{4k+2}y,\cdots,x^{6k}y\}.$$ As $V\cap VV=\varnothing$, the set $V$ is product-free. 
\end{proof}

Incidentally, we note here that our Proposition \ref{MR3} shows that the list of locally maximal product-free sets of size $4$ in $D_{14}$ given in Table $1$ of \cite{SW1974} is not correct. In particular, the authors claimed that $S=\{a,ab,a^3b,a^6b\}$ is locally maximal. However, this is not true as $S$ is contained in a product-free set of size $5$; viz. $\{a,a^{-1},ab,a^3b,a^6b\}$.

\begin{rem}\label{R6}
Observe that $G=V \sqcup VV$ in the proof of Proposition \ref{P5} above. Thus, $V$ fills $G$. By Lemma \ref{GH2009L} therefore, $V$ is a locally maximal product-free subset of $D_{12n+2}$.
\end{rem}

\noindent We give (without proof) Proposition \ref{P7} and Lemma \ref{L8}, whose proofs are similar to those in Section~3. 

\begin{pro}\label{P7}
Up to automorphisms of $D_{14}$, the only locally maximal product free set of size $5$ in $D_{14}$ is $V$, where $V$ is as defined in the proof of Proposition \ref{P5}.
\end{pro}

\begin{lem}\label{L8}
There is no locally maximal product-free set of size $6$ in $D_{14}$.
\end{lem}

\noindent The following result together with Proposition \ref{P5} disprove the stated conjecture.
\begin{thm}\label{T9}
$D_{14}$ is a filled group.
\end{thm}
\begin{proof}
From our discussion in Section~3, if $S$ is a locally maximal product-free set in $D_{14}$, then $4\leq |S|\leq 7$. If $|S|=4$, then by Theorem \ref{R2}, $S$ is mapped by an automorphism of $D_{14}$ into $W:=\{x^2,x^3,y,x^6y\}$. As $WW=\{1,x,x^4,x^5,x^6,xy,x^2y,x^3y,x^4y,x^5y\}$, the set $W$ fills $D_{14}$. By Remark \ref{R6} and Proposition \ref{P7}, any locally maximal product-free set of size $5$ in $D_{14}$ fills the group. By Lemma \ref{L8}, there is no locally maximal product-free set of size $6$ in $D_{14}$. By Proposition \ref{AH2015A_Pro}, the only locally maximal product-free set of size $7$ in $D_{14}$ is $M_1$, which by definition, fills $D_{14}$. Since every locally maximal product-free set in $D_{14}$ fills $D_{14}$, therefore $D_{14}$ is a filled group.
\end{proof}

\noindent The disproved conjecture of Street and Whitehead left us with no other known example of a dihedral group which is not a filled group. We show that such examples exist as follows:

\begin{thm}\label{last}
If $S$ is a locally maximal product-free set of size $k\geq 3$ in a finite dihedral group $G$ of order $k(k+1)$, then $S$ does not fill $G$.
\end{thm}
\begin{proof}
Suppose $S$ is a locally maximal product-free set of size $k\geq 3$ in a finite dihedral group $G$ of order $k(k+1)$. As $|SS|\leq k^2$, and $S\cap SS = \varnothing$, the set $S\sqcup SS$ fills $G$ if and only if $|S\sqcup SS|=k^2+k$. As $|S|\geq 3$, either $S$ contains two rotations or two reflections. If $S$ contains two reflections $b_1$ and $b_2$, then as $b_1^2=b_2^2=1$, we must have that $|SS|<k^2$. On the other hand, if $S$ contains two rotations $a_1$ and $a_2$, then as $a_1a_2=a_2a_1$, we must have that $|SS|<k^2$. In either case, $|S\sqcup SS|<k^2+k$; so $S$ does not fill $G$.
\end{proof}

 An example of the construction given in Theorem \ref{last} exists in $D_{20}$ as $S=\{x,x^8,y,x^5y\}$ is locally maximal in $D_{20}$ but does not fill the group. Thus, not every dihedral group is a filled group.

\begin{rem}
\noindent Street and Whitehead in \cite{SW1974} and \cite{SW1974A} pointed out that any dihedral group of order less than $14$ is a filled group. We have shown that $D_{14}$ is also filled. In fact, the first example of a non-filled dihedral group is $D_{16}$. By Theorem \ref{R2}, the set $Y:=\{x,x^6,y,x^4y\}$ is locally maximal in $D_{16}$. However, $|Rot(Y\sqcup YY)|=6<8$.
\end{rem}

We  make the following more general observation.

\begin{pro}\label{dih1} If $8$ divides $n$, then $D_{2n}$ is not filled. In particular, the only filled dihedral 2-groups are $D_4$ and $D_8$. 
\end{pro}

\begin{proof} Suppose $8$ divides $n$. Let $H = \langle x^{8}\rangle $. Then $H$ is a normal subgroup of $D_{2n}$ whose quotient is dihedral of order 16. By Lemma \ref{swl1}, and the fact that $D_{16}$ is not filled, we see that $D_{2n}$ is not filled.  
\end{proof}

We finish this paper with a table giving the known filled groups of order up to 31.

\begin{table}[h!]
\begin{center} \begin{tabular}{c|c} Order & Groups\\
\hline 2 & $C_2$\\
3 & $C_3$\\
4 & $C_2 \times C_2$\\
5 & $C_5$\\
6 & $D_6$\\
8 & $C_2^3$, $D_8$\\
10 & $D_{10}$\\
12 & $D_{12}$\\
14 & $D_{14}$\\
16 & $C_2^4$, $D_8 \times C_2$\\
22 & $D_{22}$
\end{tabular}
\caption{Filled Groups of Order less than 32}\label{table1}\end{center}
\end{table}

Table \ref{table1} was calculated using GAP, along with the results obtained in this paper. For example we only needed to check nonabelian groups of even order which have no normal subgroups of index 3. We can see that the known nonabelian filled groups are either 2-groups or dihedral (or both). The same reasoning as in Proposition \ref{dih1}, applied to the dihedral groups in Table \ref{table1}  implies that $D_{2n}$ is not filled if $n$ is divisible by $8, 9, 10, 12, 13, 14$ or $15$, but we are not able to fully classify the filled dihedral groups.

\end{document}